\documentclass[pdflatex,sn-mathphys-num]{sn-jnl}

\usepackage{graphicx}%
\usepackage{multirow}%
\usepackage{amsmath,amssymb,amsfonts}%
\usepackage{amsthm}%
\usepackage{mathrsfs}%
\usepackage[title]{appendix}%
\usepackage{xcolor}%
\usepackage{textcomp}%
\usepackage{manyfoot}%
\usepackage{booktabs}%
\usepackage{algorithm}%
\usepackage{algorithmicx}%
\usepackage{algpseudocode}%
\usepackage{listings}%
\usepackage{placeins}%
\usepackage{bm}%
\usepackage{siunitx}

\theoremstyle{thmstyleone}%
\newtheorem{theorem}{Theorem}%
\newtheorem{lemma}{Lemma}%

\theoremstyle{thmstyletwo}%
\newtheorem{remark}{Remark}%

\theoremstyle{thmstylethree}%
\DeclareMathOperator*{\argmin}{arg\,min}


\let\snabstract\abstract
\renewcommand{\abstract}[1]{\snabstract{#1\looseness=-1}}

\begin{document}

\title[CV-based feature selection for linear SVM]{Cross-validated feature selection using least-squares support vector machines for binary classification}

\author[1]{\fnm{Masaharu} \sur{Mori}}\email{s2520490@u.tsukuba.ac.jp}
\author[2]{\fnm{Shunnosuke} \sur{Ikeda}}\email{ikeda@cs.tsukuba.ac.jp}
\author[3]{\fnm{Ryuta} \sur{Tamura}}\email{r.tamura.cbc@gmail.com}
\author*[2,4]{\fnm{Yuichi} \sur{Takano}}\email{ytakano@sk.tsukuba.ac.jp}
\author[5]{\fnm{Ryuhei} \sur{Miyashiro}}\email{r-miya@cc.tuat.ac.jp}

\affil[1]{\orgdiv{Graduate School of Science and Technology}, \orgname{University of Tsukuba}, \orgaddress{\street{1--1--1 Tennodai}, \city{Tsukuba-shi}, \postcode{305--8573}, \state{Ibaraki}, \country{Japan}}}

\affil*[2]{\orgdiv{Institute of Systems and Information Engineering}, \orgname{University of Tsukuba}, \orgaddress{\street{1--1--1 Tennodai}, \city{Tsukuba-shi}, \postcode{305--8573}, \state{Ibaraki}, \country{Japan}}}

\affil[3]{\orgname{Gurobi Japan Co., Ltd.}, \orgaddress{\street{Zelkova Building 3F, 1--25--12 Fuchu-cho}, \city{Fuchu-shi}, \postcode{183--0055}, \state{Tokyo}, \country{Japan}}}

\affil[4]{\orgdiv{Center for Artificial Intelligence Research, Tsukuba Institute for Advanced Research (TIAR)}, \orgname{University of Tsukuba}, \orgaddress{\street{1--1--1 Tennodai}, \city{Tsukuba-shi}, \postcode{305--8577}, \state{Ibaraki}, \country{Japan}}}

\affil[5]{\orgdiv{Institute of Engineering}, \orgname{Tokyo University of Agriculture and Technology}, \orgaddress{\street{2--24--16 Naka-cho}, \city{Koganei-shi}, \postcode{184--8588}, \state{Tokyo}, \country{Japan}}}

\abstract{
Support vector machines (SVMs) are widely used for binary classification.
However, irrelevant or redundant features can degrade classification accuracy and model interpretability, making feature selection important for SVMs.
Most SVM feature selection methods rely on sequential search or regularization and do not directly identify the best feature subset.
Mixed-integer optimization (MIO) is a direct approach to best subset selection.
Cross-validation (CV) is broadly applicable to feature selection because it evaluates predictive performance on validation data without model assumptions.
Although Purcaru et al. (2026) proposed an MIO formulation for CV-based linear-SVM feature selection (SVM-CV), this method suffers from high computational cost.
We develop computationally efficient CV-based MIO formulations for feature selection in linear SVM classification by using least squares support vector machines (LS-SVMs).
We formulate the problem as a bilevel MIO problem, where the lower-level problem trains LS-SVM classifiers and the upper-level problem selects features based on validation sets.
We then reformulate it as a single-level MIO problem by exploiting the optimality conditions of LS-SVM.
We propose two formulations: a hinge-loss-based validation model and a least-squares-based validation model.
We compare our methods with SVM-CV and other SVM-based and logistic-regression-based methods through simulation experiments on synthetic data.
The hinge-loss-based validation model attained high feature selection accuracy under the high-noise and moderate-noise conditions, particularly with more candidate features, and the least-squares-based validation model performed well in feature selection under the low-noise condition at the lower correlation level.
Furthermore, our two models achieved substantially faster computation than the existing SVM-CV model.
}


\maketitle

\section{Introduction}\label{sec:introduction}

\subsection{Background}\label{subsec:intro-background}

Support vector machines (SVMs) are a family of machine learning methods based on optimal separating hyperplanes for binary classification, and are known to exhibit high generalization performance even in high-dimensional spaces \citep{cortes1995svm,vapnik1998statistical,cervantes2020comprehensive}.
SVMs have also been widely extended to nonlinear data analysis through kernel methods, and have attracted considerable attention in various scientific fields \citep{cervantes2020comprehensive}.
On the other hand, linear SVM classification can also be an effective option depending on the application domain and the nature of the data.
Linear SVM classification has played an important role in applications involving large-scale sparse data, such as document classification and natural language processing.
Motivated by such applications, efficient algorithms have been proposed for large-scale linear SVM classification \citep{joachims2006linearsvm,shalevshwartz2007pegasos,hsieh2008dualcoordinate,ito2017acceleratedclassification}.

Feature selection is the task of selecting a subset of features to be used for machine learning models from a set of candidate features \citep{guyon2003intro,saeys2007review,li2017dataperspective}.
Feature selection can make the relationship between predictive features and response labels easier to interpret, and can reduce data acquisition, storage, and computational costs.
It can also improve prediction performance on unobserved data by reducing overfitting through the removal of irrelevant or redundant features.
For these reasons, feature selection methods have been widely studied \citep{guyon2003intro,saeys2007review,li2017dataperspective}.
These methods are commonly classified into filter, wrapper, and embedded methods.
Filter methods rank features according to evaluation criteria before model training, with the Fisher score \citep{saeys2007review} and Relief \citep{kira1992feature} being representative examples.
Wrapper methods search for feature subsets by repeatedly training subset models, with recursive feature elimination \citep{guyon2002gene} being a representative example.
Embedded methods provide a feature subset as part of the training process, with \(L_1\)-regularization \citep{tibshirani1996lasso} being a representative example.

Various statistical criteria have been used to evaluate the quality of candidate feature subsets.
For likelihood-based statistical models, information criteria such as the Akaike information criterion (AIC) \citep{akaike1974aic} and the Bayesian information criterion (BIC) \citep{schwarz1978bic} are widely used.
AIC emphasizes the estimation of predictive performance on unobserved data, whereas BIC aims to identify the true model underlying the data-generating process.
However, the validity of these statistical criteria depends on assumptions about the predictive model and error distribution.
In contrast to these information criteria, cross-validation evaluates the predictive performance of models learned on training data by using separate validation data, mainly assuming that data instances follow the same distribution and that the training and validation datasets are independent \citep{stone1974cv,arlot2010cvsurvey}.
Therefore, the cross-validation criterion can be applied as an evaluation criterion to a wide range of predictive methods.

\subsection{Related work}\label{subsec:intro-related}

Several approaches have been developed for feature selection in linear SVM classification.
One typical approach approximated an \(L_0\)-regularization term or a cardinality constraint by a concave function \citep{bradley1998concave}.
\citet{weston2003zeronorm} approximated the \(L_0\)-norm by a sum of logarithmic functions.
\citet{lethi2008dcsvm} handled an \(L_0\)-regularization term through difference-of-convex (DC) functions and the DC algorithm.
\citet{zhu2003onenormsvm} studied \(L_1\)-regularization-based methods.
\citet{fung2004newton} developed a fast Newton method for efficient training of \(L_1\)-norm SVMs.
However, since these methods replace the original discrete best-subset selection problem with continuous approximations or regularization terms, even an optimal solution of the resulting continuous optimization problem is not guaranteed to be the best subset of features for the original discrete optimization problem.

The mixed-integer optimization (MIO) approach to feature selection was first proposed for linear regression in the 1970s \citep{arthanari1981mathematical}.
It has recently attracted renewed attention due to advances in optimization algorithms and computing environments \citep{konno2009variables,cozad2014surrogate,bertsimas2016bestsubset,ustun2016supersparse,hastie2020extended}.
Unlike many heuristic search methods, MIO-based methods can select the best subset of features based on a prescribed objective function \citep{miyashiro2015misocp,miyashiro2015mallowscp,bertsimas2016bestsubset,park2020subset,takano2020bestsubsetcv}.
MIO formulations using information criteria were also proposed for feature selection in logistic regression, allowing the number of selected features to be optimized without being specified in advance \citep{sato2016logisticmio}.
Several studies have also developed MIO-based methods for best subset selection in regression models while accounting for multicollinearity \citep{bertsimas2016algorithmic,tamura2017multicollinearity,tamura2019vif,bertsimas2020scalable,kikuchi2025twoway}.

MIO formulations have also been proposed for feature selection in linear SVM classification.
\citet{maldonado2014svm} proposed mixed-integer linear optimization (MILO) formulations with a cardinality constraint for limiting the number of selected features.
Subsequent studies have developed decomposition and relaxation approaches, including generalized Benders decomposition and Lagrangian relaxation, for solving cardinality-constrained SVM feature selection problems \citep{aytug2015benders,gaudioso2017lagrangian}.
Building on this line of work, \citet{labbe2019milpsvm} proposed a related cardinality-constrained MILO formulation and developed exact and heuristic solution procedures.
More recently, \citet{bomze2025svm} developed mixed-integer formulations, semidefinite relaxations, and exact and heuristic solution procedures for cardinality-constrained feature selection in linear SVM classification.
MIO formulations were also developed for feature selection in nonlinear kernel SVM classification.
\citet{tamura2024kernel} proposed a cardinality-constrained MILO formulation that selects a feature subset by maximizing kernel--target alignment associated with a Gaussian kernel.
However, these formulations require specifying the number of features to be selected in advance through cardinality constraints.

In contrast, \citet{purcaru2026crossvalidated} proposed a cross-validation-based mixed-integer quadratic optimization (MIQO) formulation for linear SVM classification that simultaneously optimizes the feature subset and its cardinality.
They formulated this task as a bilevel optimization problem with embedded cross-validation that minimizes the average hinge loss across folds.
They reformulated this bilevel problem as a single-level MIQO formulation by expressing the optimality of the lower-level SVM training problems through SVM duality.
However, this reformulation substantially degrades computational efficiency, and they reported that their method took considerably longer to solve than the other methods they compared.

\subsection{Our contribution}\label{subsec:intro-contribution}

The goal of this paper is to develop computationally efficient cross-validation-based MIO formulations for feature selection in linear SVM classification.
Unlike the recent cross-validation-based MIO formulation \citep{purcaru2026crossvalidated}, our formulations employ least squares support vector machines (LS-SVMs), whose optimality conditions can be written as a system of linear equations \citep{suykens1999lssvm}.
Our feature selection problem using the cross-validation criterion can be formulated as a bilevel optimization problem.
In the lower-level problem, LS-SVM classifiers are trained on training sets under a given feature subset.
In the upper-level problem, the resulting classifiers are applied to validation sets, and the feature subset is selected according to an upper-level validation criterion.
By replacing the lower-level LS-SVM training problems with their optimality conditions, we propose two types of single-level MIO formulations.
The first formulation, the hinge-loss-based validation model, minimizes the hinge loss on validation sets in the upper-level problem, and the second formulation, the least-squares-based validation model, minimizes the residual sum of squares on validation sets in the upper-level problem.

We evaluate our methods through simulation experiments with 10 and 20 candidate features, lower and higher feature-correlation levels, and three signal-to-noise ratio levels.
The synthetic datasets were generated by adapting the simulation design of \citet{purcaru2026crossvalidated} and related best-subset-selection studies \citep{bertsimas2016bestsubset,hastie2020extended,takano2020bestsubsetcv}.
We compare our methods with the cross-validation-based MIQO formulation (SVM-CV) \citep{purcaru2026crossvalidated}, \(L_1\)-regularization \citep{zhu2003onenormsvm}, recursive feature elimination \citep{guyon2002gene} for linear SVM classification, and information-criterion-based MIO models for logistic regression \citep{sato2016logisticmio}.
The hinge-loss-based validation model attained high feature selection accuracy with 20 candidate features under the high-noise and moderate-noise conditions regardless of the correlation level.
The least-squares-based validation model attained feature selection accuracy comparable to or better than that of the other SVM-based methods under the low-noise condition at the lower correlation level.
Furthermore, in the setting with 10 candidate features, both proposed models were faster than the other aforementioned MIO models, and in the setting with 20 candidate features, they were substantially faster than SVM-CV, with the least-squares-based validation model being consistently faster than the hinge-loss-based one.

\section{Linear support vector machines for binary classification}\label{sec:svm}

This section reviews the linear SVM and LS-SVM for binary classification, which are used to formulate our feature selection problem.
Throughout this paper, we denote the set of consecutive positive integers as \([n] := \{1,2,\ldots,n\}\).

\subsection{Linear support vector machines}\label{subsec:binary-svm}

We consider a classification problem that assigns a binary class label \(\hat{y} \in \{-1,+1\}\) to each input feature vector \(\bm{x} := (x_j)_{j \in [p]} \in \mathbb{R}^p\), where \(p\) denotes the number of features.
For the linear SVM, the linear classifier is defined as follows:
\begin{equation}
f(\bm{x}) = \bm{w}^\top \bm{x} + b
= \sum_{j=1}^{p} w_j x_j + b,
\label{eq:svm-decision-function}
\end{equation}
where \(\bm{w} := (w_j)_{j \in [p]} \in \mathbb{R}^p\) is a coefficient vector and \(b \in \mathbb{R}\) is the intercept.
The class label is predicted by the sign of the linear classifier:
\begin{equation}
\left\{
\begin{aligned}
f(\bm{x}) < 0 &\Rightarrow \hat{y} = -1, \\
f(\bm{x}) \ge 0 &\Rightarrow \hat{y} = +1.
\end{aligned}
\right.
\label{eq:svm-prediction}
\end{equation}

Now suppose that a dataset \(\{(\bm{x}_i,y_i)\mid i \in [n]\}\) of size \(n\) is given, where \(\bm{x}_i := (x_{ij})_{j \in [p]} \in \mathbb{R}^p\) is the feature vector and \(y_i \in \{-1,+1\}\) is the corresponding class label for each data instance \(i \in [n]\).
Correct classification requires that the predicted class label \(\hat{y}\) agree with the true class label \(y_i\) for each \(i \in [n]\).
This is the case if
\begin{equation}
y_i f(\bm{x}_i) > 0
\quad (i \in [n]).
\label{eq:linear-separable-condition}
\end{equation}
If this condition holds for all \(i \in [n]\), the two classes are linearly separable by the linear classifier in Eq.~\eqref{eq:svm-decision-function}.
In practice, however, this separability condition is violated for some \(i \in [n]\).

Accordingly, we introduce a vector \(\bm{\xi} := (\xi_i)_{i \in [n]} \in \mathbb{R}_{+}^n\) composed of decision variables to allow classification errors.
The soft-margin linear SVM minimizes the sum of the \(L_2\)-regularization term for margin maximization and the hinge loss for correct classification \citep{cortes1995svm}:
\begin{align}
\min_{b,\, \bm{w},\, \bm{\xi}} \quad
& \frac{1}{2}\|\bm{w}\|_2^2 + C \sum_{i=1}^n \xi_i
\label{eq:soft-margin-svm-objective}
\\
\textrm{s.\ t.} \quad
& y_i(\bm{w}^\top \bm{x}_i + b) \ge 1 - \xi_i
\quad (i \in [n]),
\label{eq:soft-margin-svm-constraint}
\\
& b \in \mathbb{R}, \ \bm{w} \in \mathbb{R}^p, \ \bm{\xi} \in \mathbb{R}_{+}^n,
\label{eq:soft-margin-svm-domain}
\end{align}
where \(C > 0\) is a weight parameter of classification errors.

\subsection{Least squares support vector machines}\label{subsec:ls-svm}
The least squares support vector machine (LS-SVM) introduces a vector \(\bm{e} := (e_i)_{i \in [n]} \in \mathbb{R}^n\) composed of decision variables that represent regression residuals, replaces the inequality constraints in Eq.~\eqref{eq:soft-margin-svm-constraint} with equality constraints, and minimizes the residual sum of squares \citep{suykens1999lssvm}.
Specifically, it is formulated as the following optimization problem:
\begin{align}
\min_{b,\, \bm{w},\, \bm{e}} \quad
& \frac{1}{2}\|\bm{w}\|_2^2 + \frac{\gamma}{2}\sum_{i=1}^n e_i^2
\label{eq:lssvm-primal-objective}
\\
\textrm{s.\ t.} \quad
& y_i(\bm{w}^\top \bm{x}_i + b) = 1 - e_i
\quad (i \in [n]),
\label{eq:lssvm-primal-constraint}
\\
& b \in \mathbb{R}, \ \bm{w} \in \mathbb{R}^p, \ \bm{e} \in \mathbb{R}^n,
\label{eq:lssvm-primal-domain}
\end{align}
where \(\gamma > 0\) is a weight parameter of regression residuals.

Since Eq.~\eqref{eq:soft-margin-svm-constraint} is an inequality, the slack variable \(\xi_i\) of the soft-margin SVM is positive at an optimal solution only when the margin \(y_i(\bm{w}^\top\bm{x}_i+b)\) is less than one.
In LS-SVM, by contrast, the equality constraint in Eq.~\eqref{eq:lssvm-primal-constraint} makes the residual \(e_i\) nonzero whenever the margin differs from one, and its square is minimized in Eq.~\eqref{eq:lssvm-primal-objective}.
Figure~\ref{fig:svm-lssvm-concept} illustrates this difference: SVM minimizes classification errors between two classes (Figure~\ref{fig:svm-lssvm-concept}a), whereas LS-SVM minimizes the regression residuals for each class (Figure~\ref{fig:svm-lssvm-concept}b).

\begin{figure}[t]
\centering
\includegraphics[width=\textwidth]{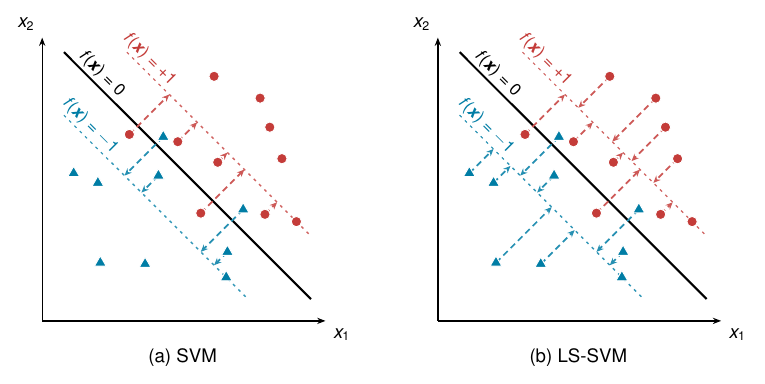}
\caption{Schematic comparison of SVM and LS-SVM}
\label{fig:svm-lssvm-concept}
\end{figure}

\subsection{Optimality conditions for LS-SVM}

LS-SVM can be reduced to a system of linear equations via the Karush--Kuhn--Tucker conditions and can also be extended to nonlinear classification using the kernel trick \citep{suykens1999lssvm}.
In this subsection, unlike in the previous study \citep{suykens1999lssvm}, we present optimality conditions of compact size that do not depend on the number of data instances.

Consider the LS-SVM problem defined by
Eqs.~\eqref{eq:lssvm-primal-objective}--\eqref{eq:lssvm-primal-domain}.
From Eq.~\eqref{eq:lssvm-primal-constraint}, the regression residual \(e_i\) is uniquely determined by \((b,\bm{w})\) as
\begin{equation}
e_i
=
1-y_i\left(\bm{w}^{\top}\bm{x}_i+b\right)
\qquad (i\in[n]).
\label{eq:lssvm-error-substitution}
\end{equation}
By substituting Eq.~\eqref{eq:lssvm-error-substitution} into the objective function, we equivalently rewrite
problem~\eqref{eq:lssvm-primal-objective}--\eqref{eq:lssvm-primal-domain}
as the following unconstrained optimization problem:
\begin{equation}
\min_{b,\,\bm{w}} \quad
\frac{1}{2}\|\bm{w}\|_2^2
+
\frac{\gamma}{2}
\sum_{i\in[n]}
\left(
1-y_i\left(\bm{w}^{\top}\bm{x}_i+b\right)
\right)^2.
\label{eq:lssvm-coef-unconstrained}
\end{equation}
The optimal solution of
problem~\eqref{eq:lssvm-coef-unconstrained}
is characterized by the following lemma.

\begin{lemma}\label{thm:lssvm-coefficient-optimality}
A pair \((b,\bm{w})\in\mathbb{R}\times\mathbb{R}^p\) is optimal for
problem~\eqref{eq:lssvm-coef-unconstrained}
if and only if
\begin{align}
b
&=
\frac{
\sum_{i\in[n]}y_i
-
\bm{w}^{\top}\sum_{i\in[n]}\bm{x}_i
}{
n
},
\label{eq:lssvm-coef-optimality-b}
\\
\bm{w}
&=
\gamma
\sum_{i\in[n]}
\left(
y_i-\bm{w}^{\top}\bm{x}_i-b
\right)
\bm{x}_i.
\label{eq:lssvm-coef-optimality-w}
\end{align}
\end{lemma}

\begin{proof}
The objective function in
Eq.~\eqref{eq:lssvm-coef-unconstrained}
is a convex quadratic function of \((b,\bm{w})\) because $\gamma>0$.
Therefore, the first-order optimality conditions are necessary and sufficient
for global optimality.
Differentiating the objective function with respect to
\(b\) and \(\bm{w}\), while noting that \(y_i^2=1\), yields
Eqs.~\eqref{eq:lssvm-coef-optimality-b}--\eqref{eq:lssvm-coef-optimality-w}.
\end{proof}

\section{Cross-validation criterion for feature selection}\label{sec:cvc}

We partition the index set \([n]\) of data instances into \(K\) validation sets \(\mathcal{V}_k\) \((k \in [K])\) of approximately equal size, satisfying the following conditions:
\begin{equation}
[n]
=
\bigcup_{k \in [K]} \mathcal{V}_k,
\qquad
\mathcal{V}_k \cap \mathcal{V}_{\ell} = \emptyset \quad (k \ne \ell),
\qquad
|\mathcal{V}_k| \approx \frac{n}{K}
\quad (k \in [K]).
\label{eq:cv-partition}
\end{equation}
For each \(k \in [K]\), the corresponding training set is defined as the complement of the validation set:
\begin{equation}
\mathcal{T}_k := [n]\setminus\mathcal{V}_k.
\label{eq:training-validation-set}
\end{equation}

We introduce a vector \(\bm{z} := (z_j)_{j \in [p]} \in \{0,1\}^p\) composed of binary decision variables for feature selection.
Here, \(z_j = 1\) means that the \(j\)-th feature is selected, whereas \(z_j = 0\) means that it is excluded from the linear classifier.
For each \(k \in [K]\) and the subset of features selected by \(\bm{z}\), the training phase learns the LS-SVM classifier from the training set \(\mathcal{T}_k\):
\vspace{-0.5\baselineskip}
\begin{align}
\left(
\hat{b}^{(k)},
\hat{\bm{w}}^{(k)}
\right)
\in
\argmin_{b^{(k)},\, \bm{w}^{(k)}} \quad
&
\frac{1}{2}\|\bm{w}^{(k)}\|_2^2
+
\frac{\gamma}{2}
\sum_{i \in \mathcal{T}_k}
\left(
1 - y_i
\left(
(\bm{w}^{(k)})^\top \bm{x}_i + b^{(k)}
\right)
\right)^2
\label{eq:cv-training-objective}
\\
\textrm{s.\ t.} \quad
&
z_j = 0
\Rightarrow
w_j^{(k)} = 0
\quad
(j \in [p]),
\label{eq:cv-training-constraint}
\end{align}
where \(b^{(k)}\in\mathbb R\) and \(\bm w^{(k)}:=(w_j^{(k)})_{j\in[p]}\in\mathbb R^p\) denote the intercept and coefficient vector, respectively, of the linear classifier in the \(k\)-th fold.
The logical implication in Eq.~\eqref{eq:cv-training-constraint} enforces the feature selection encoded by \(\bm{z}\in\{0,1\}^p\).
Using the solution \(\left(\hat{b}^{(k)}, \hat{\bm{w}}^{(k)}\right)\) obtained in the training phase, the validation phase evaluates the hinge loss on the validation sets \(\mathcal{V}_k\) \((k \in [K])\):
\begin{align}
\min_{\bm{\xi}^{(k)}} \quad
&
\sum_{k \in [K]}
\sum_{i \in \mathcal{V}_k}
\xi_i^{(k)}
\label{eq:cv-validation-objective}
\\
\textrm{s.\ t.} \quad
&
y_i
\left(
(\hat{\bm{w}}^{(k)})^\top \bm{x}_i + \hat{b}^{(k)}
\right)
\ge
1 - \xi_i^{(k)}
\quad
(i \in \mathcal{V}_k,\ k \in [K]),
\label{eq:cv-validation-constraint}
\\
&
\bm{\xi}^{(k)} \in \mathbb{R}_{+}^{|\mathcal{V}_k|}
\quad
(k \in [K]).
\label{eq:cv-validation-domain}
\end{align}
Cross-validation-based feature selection involves choosing the feature selection vector \(\bm{z} \in \{0,1\}^p\) that minimizes the total validation error in Eq.~\eqref{eq:cv-validation-objective} over all folds \(k\in[K]\).
To accomplish this directly, however, the training and validation phases must be repeated while searching over combinatorially many feature selection vectors.
Therefore, we need an optimization framework that selects \(\bm{z} \in \{0,1\}^p\) through cross-validation more efficiently.

\section{Mixed-integer optimization formulations}\label{sec:mio}

This section presents our MIO formulations for feature selection in linear SVM classification via the cross-validation criterion.

\subsection{Bilevel MIO formulation}\label{subsec:bilevel}

Cross-validation-based feature selection for linear SVMs can be posed as a bilevel MIO problem.
Specifically, the feature selection problem of minimizing the validation error can be formulated as the following upper-level problem:
\vspace{-0.5\baselineskip}
\begin{align}
\min_{b^{(k)},\, \bm{w}^{(k)},\, \bm{\xi}^{(k)},\, \bm{z}} \quad
&
\sum_{k \in [K]}
\sum_{i \in \mathcal{V}_k}
\xi_i^{(k)}
\label{eq:bilevel-upper-objective}
\\
\mathop{\textrm{s.\ t.}}\limits_{\phantom{b^{(k)},\, \bm{w}^{(k)},\, \bm{\xi}^{(k)},\, \bm{z}}} \quad
&
y_i
\left(
(\bm{w}^{(k)})^\top \bm{x}_i + b^{(k)}
\right)
\ge
1 - \xi_i^{(k)}
\quad
(i \in \mathcal{V}_k,\ k \in [K]),
\label{eq:bilevel-upper-validation}
\\
&
\left(b^{(k)}, \bm{w}^{(k)}\right)
\in
\mathcal{A}^{(k)}(\bm{z})
\quad
(k \in [K]),
\label{eq:bilevel-upper-lowerlink}
\\
&
b^{(k)} \in \mathbb{R},
\quad
\bm{w}^{(k)} \in \mathbb{R}^p,
\quad
\bm{\xi}^{(k)} \in \mathbb{R}_{+}^{|\mathcal{V}_k|}
\quad
(k \in [K]),
\notag
\\
&
\bm{z} \in \{0,1\}^p,
\label{eq:bilevel-upper-domain}
\end{align}
where the training phase of cross-validation is expressed as the following lower-level problem for each fold $k \in [K]$:
\vspace{-0.5\baselineskip}
\begin{align}
\mathcal{A}^{(k)}(\bm{z})
:=
\argmin_{b^{(k)},\, \bm{w}^{(k)}} \quad
&
\frac{1}{2}\|\bm{w}^{(k)}\|_2^2
+
\frac{\gamma}{2}
\sum_{i \in \mathcal{T}_k}
\left(
1 - y_i
\left(
(\bm{w}^{(k)})^\top \bm{x}_i + b^{(k)}
\right)
\right)^2
\label{eq:bilevel-lower-objective}
\\
\textrm{s.\ t.} \quad
&
z_j = 0
\Rightarrow
w_j^{(k)} = 0
\quad
(j \in [p]),
\label{eq:bilevel-lower-constraint}
\\
&
b^{(k)} \in \mathbb{R},
\quad
\bm{w}^{(k)} \in \mathbb{R}^p.
\label{eq:bilevel-lower-domain}
\end{align}
Here, $\mathcal{A}^{(k)}(\bm{z})$ denotes the set of LS-SVM classifiers learned on the training set $\mathcal{T}_k$ with a subset of features selected by \(\bm{z}\).

\subsection{Single-level MIO reformulation}\label{subsec:single-level}
For each \(\bm{z}\in\{0,1\}^p\), define the set of selected features by
\[
\mathcal{S}:=\{j\in[p]\mid z_j=1\}.
\]
For each \(k\in[K]\) and \(i\in[n]\), we also define the associated subvectors as
\[
\bm{w}_{\mathcal{S}}^{(k)}
:=
\bigl(w_j^{(k)}\bigr)_{j\in\mathcal{S}}
\in\mathbb{R}^{|\mathcal{S}|},
\qquad
\bm{x}_{i\mathcal{S}}
:=
\bigl(x_{ij}\bigr)_{j\in\mathcal{S}}
\in\mathbb{R}^{|\mathcal{S}|}.
\]

Bilevel optimization problems are generally hard to solve directly \citep{colson2007bilevel,sinha2018bilevel}.
To resolve this difficulty, we replace the lower-level problem with its optimality conditions, which can be written as linear equality constraints.
As a first step, the following theorem demonstrates the uniqueness of the optimal solution to the lower-level problem.

\begin{theorem}\label{thm:lssvm-lower-uniqueness}
For each \(k\in[K]\) and \(\bm{z}\in\{0,1\}^p\), the lower-level
problem~\eqref{eq:bilevel-lower-objective}--\eqref{eq:bilevel-lower-domain}
has a unique optimal solution.
\end{theorem}

\begin{proof}
For a fixed \(\bm{z}\in\{0,1\}^p\), Eq.~\eqref{eq:bilevel-lower-constraint} implies that \(w_j^{(k)}=0\) for all \(j\in[p]\setminus\mathcal{S}\).
Therefore, the lower-level problem~\eqref{eq:bilevel-lower-objective}--\eqref{eq:bilevel-lower-domain} reduces to an unconstrained optimization problem of the same form as problem~\eqref{eq:lssvm-coef-unconstrained}, with the training set \(\mathcal{T}_k\) and the feature subset
\(\mathcal{S}\):
\begin{equation}
\min_{b^{(k)},\,\bm{w}^{(k)}_{\mathcal{S}}} \quad
\frac{1}{2}\|\bm{w}^{(k)}_{\mathcal{S}}\|_2^2
+
\frac{\gamma}{2}
\sum_{i\in\mathcal{T}_k}
\left(
1-y_i
\left(
(\bm{w}^{(k)}_{\mathcal{S}})^\top
\bm{x}_{i\mathcal{S}}
+b^{(k)}
\right)
\right)^2.
\label{eq:lssvm-optimality-objective}
\end{equation}
The Hessian
of this objective function with respect to
\(\left(b^{(k)},\bm{w}^{(k)}_{\mathcal{S}}\right)\) is
\[
\bm{H}
:=
\begin{bmatrix}
\gamma|\mathcal{T}_k|
&
\gamma
\sum_{i\in\mathcal{T}_k}
\bm{x}_{i\mathcal{S}}^\top
\\
\gamma
\sum_{i\in\mathcal{T}_k}
\bm{x}_{i\mathcal{S}}
&~
\bm{I}_{|\mathcal{S}|}
+
\gamma
\sum_{i\in\mathcal{T}_k}
\bm{x}_{i\mathcal{S}}
\bm{x}_{i\mathcal{S}}^\top
\end{bmatrix} \in \mathbb{R}^{(|\mathcal{S}|+1)\times(|\mathcal{S}|+1)},
\]
where \(\bm{I}_{|\mathcal{S}|}\) is the identity matrix of size \(|\mathcal{S}|\).
For any 
\[
\tilde{\bm{u}}
:=
\begin{bmatrix}
u_0\\
\bm{u}
\end{bmatrix}
\in
\mathbb{R}^{|\mathcal{S}|+1}
\setminus\{\bm{0}\},
\]
it follows from $\gamma>0$ that
\[
\tilde{\bm{u}}^\top
\bm{H}
\tilde{\bm{u}}
=
\|\bm{u}\|_2^2
+
\gamma
\sum_{i\in\mathcal{T}_k}
\left(
u_0+\bm{u}^\top\bm{x}_{i\mathcal{S}}
\right)^2
>0.
\]
Hence, \(\bm{H}\) is positive definite, and the objective function in Eq.~\eqref{eq:lssvm-optimality-objective} is
strongly convex.
Therefore, the lower-level problem~\eqref{eq:bilevel-lower-objective}--\eqref{eq:bilevel-lower-domain}
has a unique optimal solution.
\end{proof}

To convert the bilevel problem into a single-level problem, we impose,
for each fold \(k\in[K]\), the full-model optimality conditions in
Lemma~\ref{thm:lssvm-coefficient-optimality} on the subset model:
\begin{align}
&
b^{(k)}
=
\frac{
\sum_{i\in\mathcal{T}_k}y_i
-
(\bm{w}^{(k)})^\top
\sum_{i\in\mathcal{T}_k}\bm{x}_i
}{
|\mathcal{T}_k|
},
\label{eq:lssvm-optimality-b}
\\
&
z_j = 1
\Rightarrow
w_j^{(k)}
=
\gamma
\sum_{i\in\mathcal{T}_k}
\left(
y_i-(\bm{w}^{(k)})^\top\bm{x}_i-b^{(k)}
\right)
x_{ij}
\quad
(j\in[p]).
\label{eq:lssvm-optimality-selected-w}
\end{align}
These conditions extend the normal-equation-based constraints used in
linear regression \citep{takano2020bestsubsetcv} to LS-SVM classification.
The following theorem establishes that imposing
Eqs.~\eqref{eq:lssvm-optimality-b}--\eqref{eq:lssvm-optimality-selected-w}
is equivalent to solving the lower-level problem.

\begin{theorem}\label{thm:lssvm-optimality-equivalence}
Suppose that \((b^{(k)}, \bm{w}^{(k)}, \bm{z}) \in \mathbb{R} \times \mathbb{R}^p \times \{0,1\}^p\) satisfies Eq.~\eqref{eq:bilevel-lower-constraint}.
Then, \((b^{(k)},\bm{w}^{(k)}) \in \mathcal{A}^{(k)}(\bm{z})\) if and only if Eqs.~\eqref{eq:lssvm-optimality-b}--\eqref{eq:lssvm-optimality-selected-w} hold.
\end{theorem}
\begin{proof}
Without loss of generality, we may partition $\bm{w}^{(k)}$ such that
\[
\bm{w}^{(k)}
=
\begin{pmatrix}
\bm{w}^{(k)}_{\mathcal{S}} \\
\bm{0}
\end{pmatrix},
\qquad
\]
due to Eq.~\eqref{eq:bilevel-lower-constraint}.
Therefore, Eqs.~\eqref{eq:lssvm-optimality-b}--\eqref{eq:lssvm-optimality-selected-w} can be rewritten as
\begin{align*}
&
b^{(k)}
=
\frac{
\sum_{i \in \mathcal{T}_k} y_i
- (\bm{w}^{(k)}_{\mathcal{S}})^\top \sum_{i \in \mathcal{T}_k} \bm{x}_{i\mathcal{S}}
}{
|\mathcal{T}_k|
},
\\
&
\bm{w}^{(k)}_{\mathcal{S}}
=
\gamma
\sum_{i \in \mathcal{T}_k}
\left(
y_i - (\bm{w}^{(k)}_{\mathcal{S}})^\top \bm{x}_{i\mathcal{S}} - b^{(k)}
\right)
\bm{x}_{i\mathcal{S}},
\end{align*}
which correspond to the optimality conditions in Eqs.~\eqref{eq:lssvm-coef-optimality-b}--\eqref{eq:lssvm-coef-optimality-w} of the subset LS-SVM involving \((\mathcal{T}_k, \mathcal{S})\).
It then follows that \((b^{(k)}, \bm{w}^{(k)}_{\mathcal{S}})\) is the optimal solution of problem~\eqref{eq:lssvm-optimality-objective}, or equivalently, \((b^{(k)}, \bm{w}^{(k)}) \in \mathcal{A}^{(k)}(\bm{z})\) holds.
\end{proof}
From Theorems~\ref{thm:lssvm-lower-uniqueness} and~\ref{thm:lssvm-optimality-equivalence}, the bilevel MIO problem~\eqref{eq:bilevel-upper-objective}--\eqref{eq:bilevel-lower-domain} can be rewritten as the following single-level MILO problem:
\vspace{-0.5\baselineskip}
\begin{align}
\min_{b^{(k)},\, \bm{w}^{(k)},\, \bm{\xi}^{(k)},\, \bm{z}} \quad
&
\sum_{k \in [K]}
\sum_{i \in \mathcal{V}_k}
\xi_i^{(k)}
\label{eq:single-level-objective}
\\
\mathop{\textrm{s.\ t.}}\limits_{\phantom{b^{(k)},\, \bm{w}^{(k)},\, \bm{\xi}^{(k)},\, \bm{z}}} \quad
&
y_i
\left(
(\bm{w}^{(k)})^\top \bm{x}_i + b^{(k)}
\right)
\ge
1 - \xi_i^{(k)}
\quad
(i \in \mathcal{V}_k,\ k \in [K]),
\label{eq:single-level-validation}
\\
&
z_j = 0
\Rightarrow
w_j^{(k)} = 0
\quad
(j \in [p],\ k \in [K]),
\label{eq:single-level-selection}
\\
&
z_j = 1
\Rightarrow
w_j^{(k)}
=
\gamma
\sum_{i \in \mathcal{T}_k}
\left(
y_i
- (\bm{w}^{(k)})^\top \bm{x}_i
- b^{(k)}
\right)
x_{ij}
\quad
(j \in [p],\ k \in [K]),
\label{eq:single-level-fo-w}
\\
&
b^{(k)}
=
\frac{
\sum_{i \in \mathcal{T}_k} y_i
- (\bm{w}^{(k)})^\top \sum_{i \in \mathcal{T}_k} \bm{x}_i
}{
|\mathcal{T}_k|
}
\quad
(k \in [K]),
\label{eq:single-level-fo-b}
\\
&
b^{(k)} \in \mathbb{R},
\quad
\bm{w}^{(k)} \in \mathbb{R}^p,
\quad
\bm{\xi}^{(k)} \in \mathbb{R}_{+}^{|\mathcal{V}_k|}
\quad
(k \in [K]),
\notag
\\
&
\bm{z} \in \{0,1\}^p.
\label{eq:single-level-domain}
\end{align}
Since this reformulation eliminates the nested lower-level optimization problem, the single-level MILO problem~\eqref{eq:single-level-objective}--\eqref{eq:single-level-domain} can be solved directly by standard optimization solvers.

An analogous argument holds when the upper-level objective is replaced with the LS-SVM-based residual sum of squares (cf. Eq.~\eqref{eq:lssvm-coef-unconstrained}) on the validation sets, yielding the following single-level MIQO problem:
\vspace{-0.5\baselineskip}
\begin{align}
\min_{b^{(k)},\, \bm{w}^{(k)},\, \bm{z}} \quad
&
\sum_{k \in [K]}
\sum_{i \in \mathcal{V}_k}
\left(1 - y_i\left((\bm{w}^{(k)})^\top \bm{x}_i + b^{(k)}\right)\right)^2
\label{eq:single-level-lssvm-objective}
\\
\mathop{\textrm{s.\ t.}}\limits_{\phantom{b^{(k)},\, \bm{w}^{(k)},\, \bm{z}}} \quad
&
z_j = 0
\Rightarrow
w_j^{(k)} = 0
\quad
(j \in [p],\ k \in [K]),
\label{eq:single-level-lssvm-selection}
\\
&
z_j = 1
\Rightarrow
w_j^{(k)}
=
\gamma
\sum_{i \in \mathcal{T}_k}
\left(
y_i - (\bm{w}^{(k)})^\top \bm{x}_i - b^{(k)}
\right)
x_{ij}
\quad
(j \in [p],\ k \in [K]),
\label{eq:single-level-lssvm-fo-w}
\\
&
b^{(k)}
=
\frac{
\sum_{i \in \mathcal{T}_k} y_i
- (\bm{w}^{(k)})^\top \sum_{i \in \mathcal{T}_k} \bm{x}_i
}{
|\mathcal{T}_k|
}
\quad
(k \in [K]),
\label{eq:single-level-lssvm-fo-b}
\\
&
b^{(k)} \in \mathbb{R},
\quad
\bm{w}^{(k)} \in \mathbb{R}^p
\quad
(k \in [K]),
\notag
\\
&\bm{z} \in \{0,1\}^p.
\label{eq:single-level-lssvm-domain}
\end{align}
The size of the MIQO problem~\eqref{eq:single-level-lssvm-objective}--\eqref{eq:single-level-lssvm-domain} is independent of the number of data instances, so it is expected to be computationally efficient.

\begin{remark}
The same reformulation framework may also be adapted to support vector regression (SVR) feature selection.
For example, if the lower-level problem is replaced by ridge regression and the upper-level validation criterion is replaced by the \(\varepsilon\)-insensitive loss used in SVR, the resulting formulation becomes a minor modification of the cross-validation-based feature selection framework for linear regression \citep{takano2020bestsubsetcv}.
This extension may be useful for robust feature selection because the \(\varepsilon\)-insensitive loss evaluated on validation sets is less sensitive to outliers than the squared loss.
\end{remark}

\section{Simulation experiments}\label{sec:experiments}
This section evaluates the effectiveness of our feature selection methods through numerical experiments on synthetic datasets.
All experiments were conducted on macOS with an Apple M3 Pro central processing unit (CPU, 12 cores) and 36\,GB of memory.
All methods were implemented in Python 3.13.
\subsection{Synthetic datasets}\label{subsec:design}

We evaluated the feature selection methods on synthetic datasets with controlled signal-to-noise ratio (SNR) levels.
For each trial, we generated a synthetic dataset of size \(N=20000\).
We considered two problem settings: one with \(p=10\) candidate features and a ground-truth support size of \(\theta=5\), and the other with \(p=20\) candidate features and a ground-truth support size of \(\theta=10\).
We used \(n=100\) instances as the training set and the remaining \(N-n=19900\) instances as the testing set.
For both problem settings, we considered three SNR levels, \(\{0.25, 1.00, 4.00\}\).
We generated the synthetic datasets following prior work \citep{bertsimas2016bestsubset,hastie2020extended,takano2020bestsubsetcv,purcaru2026crossvalidated}, using the procedure described below.
We repeated this procedure 100 times for \(p=10\) and 5 times for \(p=20\).

\begin{enumerate}

\item
For each trial, we randomly selected a ground-truth support set \(S^\star\subseteq[p]\) of size \(\theta\) and generated a coefficient vector \(\widetilde{\bm{w}}^\star:= (\widetilde{w}^\star_j)_{j \in [p]}\) and an intercept \(b^\star \in \mathbb{R}\) as follows:
\[
\widetilde{w}^{\star}_j
\sim
\mathcal{U}(-1,1)
\quad (j\in S^\star),
\qquad
\widetilde{w}^{\star}_j=0
\quad (j\notin S^\star),
\qquad
b^\star\sim\mathcal{U}(-1,1).
\]
Here, \(\mathcal{U}(a,b)\) denotes the continuous uniform distribution on the interval \([a,b]\).
The ground-truth coefficient vector was then normalized as
\[
\bm{w}^\star
:=
\frac{\widetilde{\bm{w}}^\star}
{\|\widetilde{\bm{w}}^\star\|_2}.
\]

\item
The covariance matrix is \(\bm{\Sigma}_{\rho}:=(\sigma_{j\ell})_{(j,\ell)\in[p]\times[p]}\), where \(\sigma_{j\ell}:=\rho^{|j-\ell|}\) for \(j,\ell\in[p]\).
We considered two correlation levels, \(\rho\in\{0.35,0.70\}\).
Each feature vector was independently generated as \(\bm{x}_i:=\bm{L}_{\rho}\bm{u}_i\) for \(i\in[N]\), where \(\bm{u}_i:=(u_{ij})_{j\in[p]}\) was generated as \(u_{ij}\sim\mathcal{U}(-\sqrt{3},\sqrt{3})\) for \(j\in[p]\), and \(\bm{L}_{\rho}\) is the Cholesky factor satisfying \(\bm{\Sigma}_{\rho}=\bm{L}_{\rho}\bm{L}_{\rho}^{\top}\).

\item
Let \(\mathcal{N}(0,\sigma^2)\) denote the normal distribution with mean zero and variance \(\sigma^2\).
For a noise term \(\varepsilon \sim \mathcal{N}(0,\sigma^2)\), the SNR is defined as
\[
\mathrm{SNR}
:=
\frac{\mathrm{Var}((\bm{w}^{\star})^\top\bm{x})}{\mathrm{Var}(\varepsilon)}
=
\frac{(\bm{w}^{\star})^\top\bm{\Sigma}_{\rho} \bm{w}^{\star}}{\sigma^2}.
\]
For each SNR level, we set the noise variance \(\sigma^2\) according to the above definition of SNR.

\item
For each \(i \in [N]\), we generated an output of the ground-truth linear classifier as
\[
t_i
:=
(\bm{w}^\star)^\top\bm{x}_i
+
b^\star
+
\varepsilon_i,
\quad
\varepsilon_{i} \sim \mathcal{N}(0,\sigma^2).
\]

\item
The binary class label was defined as
\[
y_i
:=
\begin{cases}
-1 & (t_i<0),\\
+1 & (t_i\ge 0).
\end{cases}
\]

\end{enumerate}

\subsection{Performance metrics}\label{subsec:metrics}

Let \(\hat{\bm{z}} \in \{0,1\}^p\) denote the feature selection vector returned by each method.
For each method, let \(\hat{f}(\bm{x}) := \hat{\bm{w}}^\top \bm{x} + \hat{b}\) denote the trained model, where \(\hat{\bm{w}} \in \mathbb{R}^{p}\) and \(\hat{b} \in \mathbb{R}\) are the estimated coefficient vector and intercept term.

We used the area under the receiver operating characteristic curve (AUC) as a metric for classification performance \citep{fawcett2006roc}.
Let \(\mathcal{P}\) and \(\mathcal{Q}\) denote the index sets of positive and negative testing instances, respectively.
The empirical AUC is the proportion of positive--negative instance pairs in which the positive instance receives a higher score, with ties counted as one half:
\[
\mathrm{AUC}
:=
\frac{1}{|\mathcal{P}|\,|\mathcal{Q}|}
\sum_{i \in \mathcal{P}}
\sum_{\ell \in \mathcal{Q}}
\left[
\mathbb{I}\{\hat{f}(\bm{x}_i)>\hat{f}(\bm{x}_{\ell})\}
+
\frac{1}{2}\mathbb{I}\{\hat{f}(\bm{x}_i)=\hat{f}(\bm{x}_{\ell})\}
\right],
\]
where \(\mathbb{I}\{\cdot\}\) denotes the indicator function (i.e., \(\mathbb{I}\{A\} = 1\) if the proposition $A$ is true, and \(\mathbb{I}\{A\} = 0\) otherwise).

To evaluate feature selection accuracy, we used the F1 score
\citep{powers2011evaluation}.
Let \(\bm{z}^{\star}\in\{0,1\}^{p}\) denote the ground-truth feature
selection vector, where \(z_j^{\star}=1\) if \(j\in S^{\star}\) and
\(z_j^{\star}=0\) otherwise.
Then, \((\bm{z}^{\star})^\top\hat{\bm{z}}\),
\(\bm{1}^\top\hat{\bm{z}}\), and
\(\bm{1}^\top\bm{z}^{\star}\) represent the number of correctly selected
features, the number of selected features, and the number of ground-truth
features, respectively.
The F1 score is the harmonic mean of precision and recall:
\[
\mathrm{F1~score}
:=
\frac{2 \cdot \mathrm{Precision} \cdot \mathrm{Recall}}{\mathrm{Precision}+\mathrm{Recall}},
\]
where
\[
\mathrm{Precision}
:=
\frac{(\bm{z}^{\star})^\top \hat{\bm{z}}}{\bm{1}^\top \hat{\bm{z}}},
\qquad
\mathrm{Recall}
:=
\frac{(\bm{z}^{\star})^\top \hat{\bm{z}}}{\bm{1}^\top \bm{z}^{\star}}.
\]

We also evaluated the size of the selected feature subset:
\[
\mathrm{Subset~size}:=\bm{1}^\top \hat{\bm{z}}.
\]

\subsection{Methods under comparison}\label{subsec:simenv}

We compared the following methods for feature selection:

\begin{itemize}
\item
\textbf{Proposed-HL} (hinge-loss-based validation model): our cross-validation-based MILO formulation~\eqref{eq:single-level-objective}--\eqref{eq:single-level-domain}, which minimizes the validation hinge loss, for feature selection in linear SVMs;
\item
\textbf{Proposed-LS} (least-squares-based validation model): our cross-validation-based MIQO formulation~\eqref{eq:single-level-lssvm-objective}--\eqref{eq:single-level-lssvm-domain}, which minimizes the validation residual sum of squares, for feature selection in linear SVMs;
\item
\textbf{SVM-CV}: cross-validation-based MIQO formulation for feature selection in linear SVMs \citep{purcaru2026crossvalidated};
\item
\textbf{SVM-L1}: \(L_1\)-regularized estimation for linear SVMs \citep{zhu2003onenormsvm};
\item
\textbf{SVM-RFE}: recursive feature elimination for linear SVMs \citep{guyon2002gene};
\item
\textbf{LogReg-AIC}: AIC-based MILO formulation for logistic regression feature selection \citep{sato2016logisticmio};
\item
\textbf{LogReg-BIC}: BIC-based MILO formulation for logistic regression feature selection \citep{sato2016logisticmio}.

\end{itemize}

For Proposed-HL, Proposed-LS, and SVM-CV, we used \(K = 5\) folds for cross-validation.
For each of Proposed-HL and Proposed-LS, the weight parameter of regression residuals was chosen from \(\gamma \in \{100, 300, 500, 700, 1000\}\), whereas for SVM-CV, the weight parameter of classification errors was chosen from \(C \in \{0.001, 0.01, 0.1, 1, 10, 100, 1000\}\).
Following \citet{purcaru2026crossvalidated}, the weight parameter \(C\) of SVM-CV was divided by the number of training instances in each fold.
In each case, the parameter value that minimized the respective internal cross-validation objective on the training set was selected.
For SVM-CV, no cardinality restriction was imposed.
The Proposed-HL, Proposed-LS, and SVM-CV problems were solved with Gurobi Optimizer 12.0.3 \citep{gurobi2026}, with a time limit of 300 seconds for each solver run.
The logical implications in the Proposed-HL and Proposed-LS formulations were imposed using the indicator constraints implemented in Gurobi, whereas those in the SVM-CV formulation were represented using the big-\(M\) constraints of \citet{purcaru2026crossvalidated}.

For SVM-L1 and SVM-RFE, we used scikit-learn 1.7.2 \citep{scikit-learn}.
For SVM-L1, we used the \texttt{LinearSVC} class from \texttt{sklearn.svm} with \texttt{penalty="l1"}, \texttt{loss="squared\_hinge"}, and \texttt{dual=False}. 
The squared hinge loss was selected because \texttt{LinearSVC} does not support combining \(L_1\)-regularization with the standard hinge loss.
For SVM-RFE, we used the \texttt{RFECV} class from \texttt{sklearn.feature\_selection} with a linear \texttt{SVC} estimator and eliminated one feature at each iteration using \texttt{step=1}.
For both methods, the weight parameter of classification errors was selected from \(C\in\{0.001,0.01,0.1,1,10,100,1000\}\).
For SVM-L1, we computed the classification accuracy for each candidate value of \(C\) using \texttt{cross\_val\_score} with 5-fold stratified cross-validation and selected the value that yielded the highest validation accuracy.
For SVM-RFE, \texttt{RFECV} performed 5-fold cross-validation for each candidate value of \(C\) to determine the number of selected features and the corresponding classification accuracy.
We then selected the value of \(C\) that yielded the highest validation accuracy.

For the LogReg-AIC and LogReg-BIC methods, we used the MILO formulation of \citet{sato2016logisticmio}, which approximates the logistic loss function by a piecewise-linear function constructed from 17 tangent lines.
These MILO problems were solved with Gurobi Optimizer 12.0.3 under the same 300-second time limit.

For all methods, the AUC was computed for a linear SVM refitted on the selected features.
Specifically, we solved problem~\eqref{eq:soft-margin-svm-objective}--\eqref{eq:soft-margin-svm-domain} with Gurobi Optimizer 12.0.3 on the training set restricted to the selected features and computed the AUC on the testing set.
In this problem, we selected the weight parameter of classification errors from \(C\in\{0.001,0.01,0.1,1,10,100,1000\}\) by 5-fold stratified cross-validation on the training set, choosing the value that yielded the highest validation accuracy.

\begin{figure}[!t]
\centering
\includegraphics[width=\textwidth]{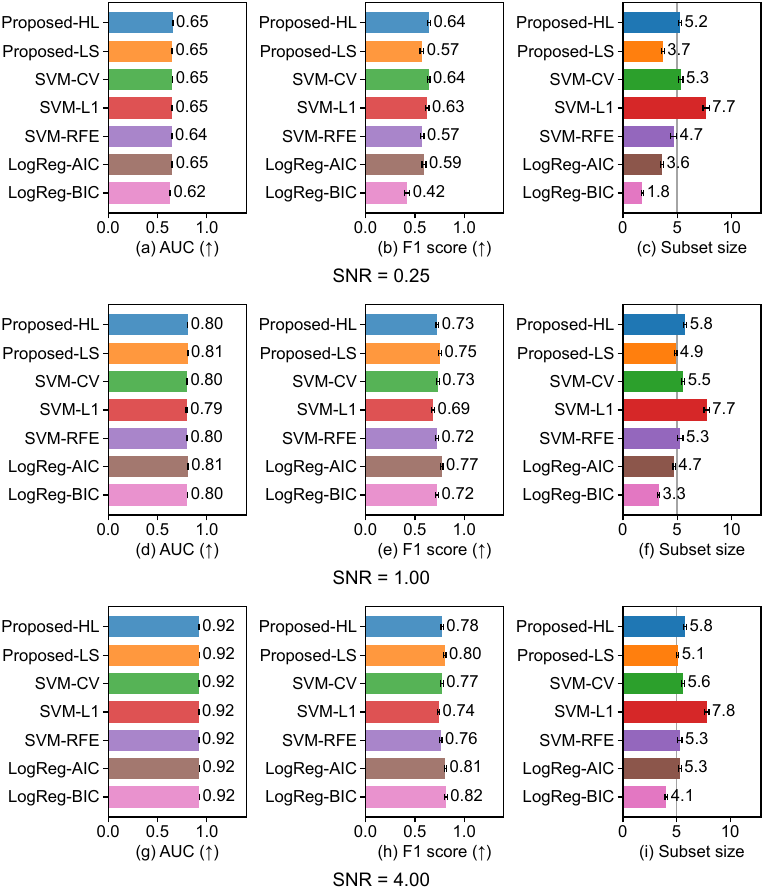}
\caption{Performance of feature selection for \(\rho=0.35\) (\(p=10\) and \(\theta=5\))}
\label{fig:p10_θ5_corr_0.35-results}
\label{fig:p10_θ5_corr_0.35-results-snr025}
\label{fig:p10_θ5_corr_0.35-results-snr10}
\label{fig:p10_θ5_corr_0.35-results-snr40}
\end{figure}

\subsection{Results}\label{subsec:results}

Figures~\ref{fig:p10_θ5_corr_0.35-results}--\ref{fig:p20_θ10_corr_0.7-results} show the feature selection performance with \(p=10\) and \(20\) candidate features (ground-truth support sizes of \(\theta=5\) and \(10\), respectively) under two correlation levels, \(\rho\in\{0.35,0.70\}\).
All results were averaged over 100 and 5 independent runs for \(p=10\) and \(p=20\), respectively, and the standard errors are shown as error bars in the figures.
Higher AUC values and F1 scores indicate better performance, and subset sizes closer to \(\theta\) are preferable.

Figure~\ref{fig:p10_θ5_corr_0.35-results} shows the results for \(p=10\) and \(\rho=0.35\).
Under the high-noise condition (\(\mathrm{SNR}=0.25\)), Proposed-HL tied for the best AUC and tied with SVM-CV for the best F1 score while selecting a number of features close to the ground-truth support size.
Under the moderate-noise and low-noise conditions (\(\mathrm{SNR}=1.00\) and \(4.00\)), LogReg-AIC and LogReg-BIC attained the highest F1 scores, respectively.
A possible explanation is that the lower noise made the sparse linear signal more distinct, making the fit--complexity trade-off imposed by AIC and BIC more effective for recovering the ground-truth support.
Among the SVM-based methods, Proposed-LS achieved the highest F1 scores under these conditions.

\begin{figure}[!t]
\centering
\includegraphics[width=\textwidth]{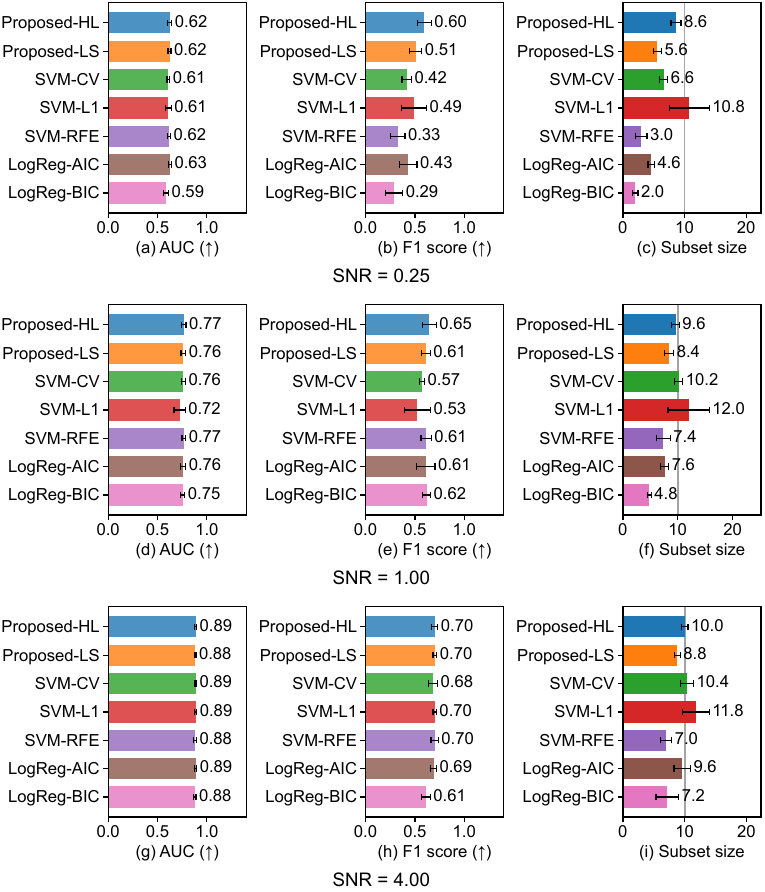}
\caption{Performance of feature selection for \(\rho=0.35\) (\(p=20\) and \(\theta=10\))}
\label{fig:p20_θ10_corr_0.35-results}
\label{fig:p20_θ10_corr_0.35-results-snr025}
\label{fig:p20_θ10_corr_0.35-results-snr10}
\label{fig:p20_θ10_corr_0.35-results-snr40}
\end{figure}

Figure~\ref{fig:p20_θ10_corr_0.35-results} shows the results for \(p=20\) and \(\rho=0.35\).
Under the high-noise and moderate-noise conditions, Proposed-HL achieved the highest F1 scores while selecting slightly fewer features than the ground-truth support size, and it also tied for the best AUC under the moderate-noise and low-noise conditions.
Under the low-noise condition, the two proposed methods, SVM-RFE, and SVM-L1 tied for the best F1 score.

\begin{figure}[!t]
\centering
\includegraphics[width=\textwidth]{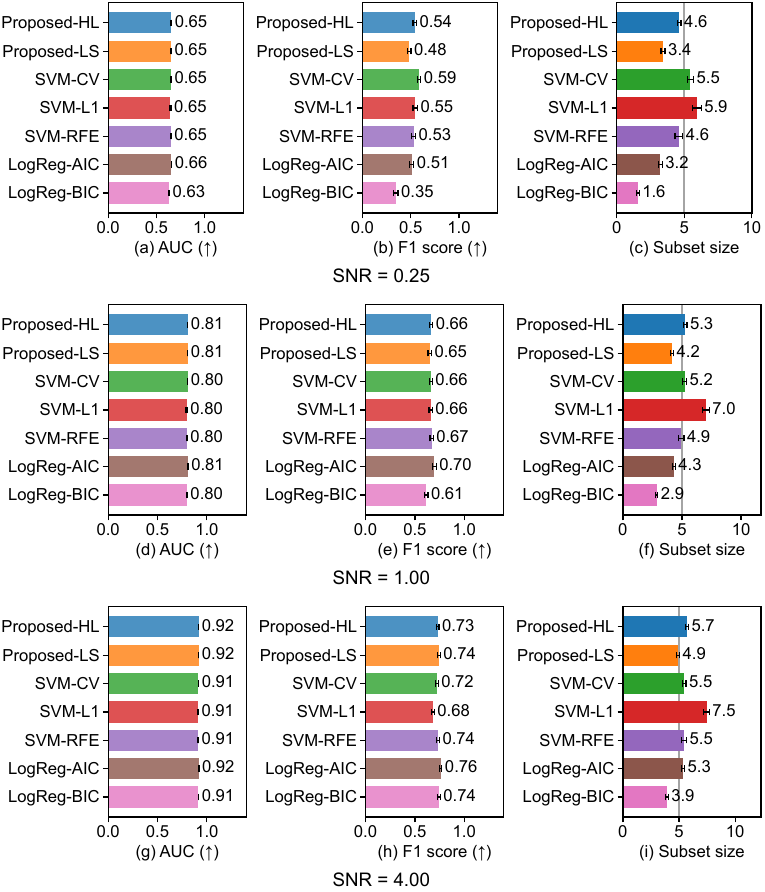}
\caption{Performance of feature selection for \(\rho=0.70\) (\(p=10\) and \(\theta=5\))}
\label{fig:p10_θ5_corr_0.7-results}
\label{fig:p10_θ5_corr_0.7-results-snr025}
\label{fig:p10_θ5_corr_0.7-results-snr10}
\label{fig:p10_θ5_corr_0.7-results-snr40}
\end{figure}

Figure~\ref{fig:p10_θ5_corr_0.7-results} shows the results for \(p=10\) and \(\rho=0.70\).
In terms of the F1 score, SVM-CV performed best under the high-noise condition, and LogReg-AIC performed best under the moderate-noise and low-noise conditions.
Compared with \(\rho=0.35\), all methods attained lower F1 scores, whereas their AUC values were almost unchanged.
This may be because the ground-truth features are more difficult to identify at the higher correlation level, where correlated features can yield similar goodness of fit or predictive performance.
The consistent use of the hinge loss for training and feature selection in SVM-CV may have worked effectively under the high-noise condition.

\begin{figure}[!t]
\centering
\includegraphics[width=\textwidth]{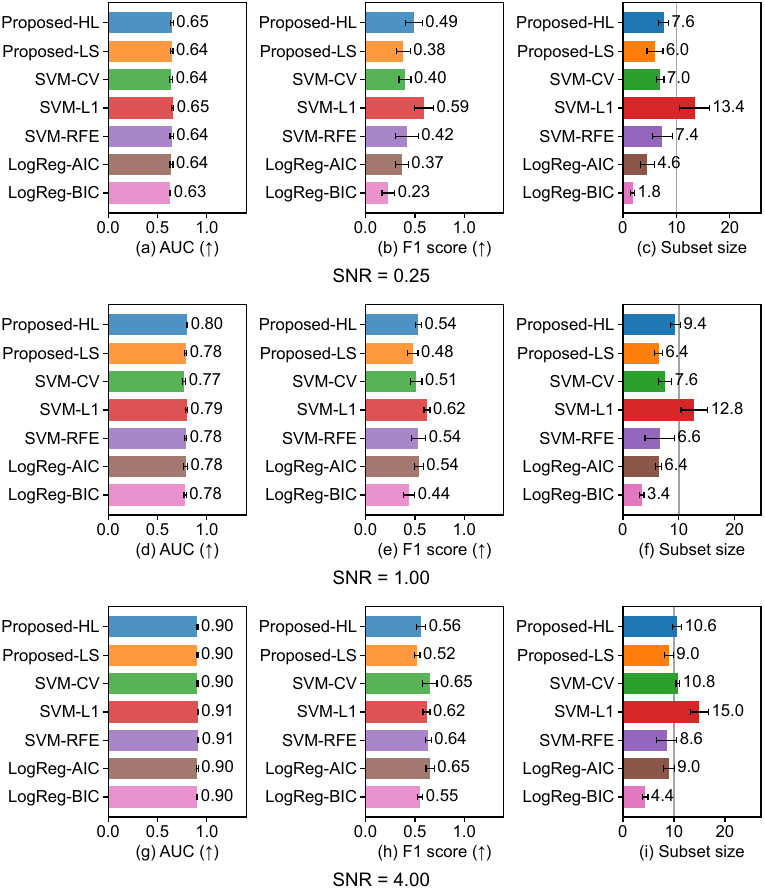}
\caption{Performance of feature selection for \(\rho=0.70\) (\(p=20\) and \(\theta=10\))}
\label{fig:p20_θ10_corr_0.7-results}
\label{fig:p20_θ10_corr_0.7-results-snr025}
\label{fig:p20_θ10_corr_0.7-results-snr10}
\label{fig:p20_θ10_corr_0.7-results-snr40}
\end{figure}

Figure~\ref{fig:p20_θ10_corr_0.7-results} shows the results for \(p=20\) and \(\rho=0.70\).
Under the high-noise and moderate-noise conditions, SVM-L1 attained the highest F1 scores, but it tended to select more features (13.4 and 12.8 on average, respectively) than the ground-truth support size.
In contrast, Proposed-HL attained or tied for the second-highest F1 scores while selecting fewer features than SVM-L1, indicating that it performed well in feature selection under these conditions.
It also attained or tied for the highest AUC values.
Under the low-noise condition, SVM-CV and LogReg-AIC tied for the best F1 score.

{\setlength{\intextsep}{4pt}
\begin{table}[!t]
\centering
\caption{Mean computation times (in seconds), with standard errors in parentheses, for each combination of \(p\), \(\theta\), \(\rho\), and SNR}
\label{tab:computation-time}
\label{tab:computation-time-p10-theta5-rho}
\label{tab:computation-time-p20-theta10-rho}
\normalsize
\setlength{\tabcolsep}{2pt}
\renewcommand{\arraystretch}{1.00}

\begin{tabular}{@{}ccc@{\hspace{1em}}lS[table-format=3.2,table-text-alignment=right]@{\,(\ensuremath{\pm}\,}S[table-format=2.2]@{)\enspace}S[table-format=3.2,table-text-alignment=right]@{\,(\ensuremath{\pm}\,}S[table-format=2.2]@{)\enspace}S[table-format=3.2,table-text-alignment=right]@{\,(\ensuremath{\pm}\,}S[table-format=2.2]@{)}@{}}
\toprule
\multirow{2}{*}{\(p\)} & \multirow{2}{*}{\(\theta\)} & \multirow{2}{*}{\(\rho\)} & \multirow{2}{*}{\textbf{Method}} & \multicolumn{6}{c}{\textbf{SNR}} \\
\cmidrule(lr){5-10}
 & & & & \multicolumn{2}{c}{0.25} & \multicolumn{2}{c}{1.00} & \multicolumn{2}{c}{4.00} \\
\midrule
10 & 5 & 0.35 & Proposed-HL & 0.06 & 0.00 & 0.04 & 0.00 & 0.04 & 0.00 \\
 &  &  & Proposed-LS & 0.05 & 0.00 & 0.04 & 0.00 & 0.04 & 0.00 \\
 &  &  & SVM-CV & 0.96 & 0.06 & 1.49 & 0.12 & 1.92 & 0.11 \\
 &  &  & SVM-L1 & {\(<0.01\)} & 0.00 & {\(<0.01\)} & 0.00 & {\(<0.01\)} & 0.00 \\
 &  &  & SVM-RFE & 0.10 & 0.04 & 0.05 & 0.00 & 0.08 & 0.02 \\
 &  &  & LogReg-AIC & 0.15 & 0.00 & 0.16 & 0.00 & 0.14 & 0.00 \\
 &  &  & LogReg-BIC & 0.15 & 0.00 & 0.17 & 0.00 & 0.14 & 0.00 \\
\midrule
10 & 5 & 0.70 & Proposed-HL & 0.06 & 0.00 & 0.05 & 0.00 & 0.04 & 0.00 \\
 &  &  & Proposed-LS & 0.05 & 0.00 & 0.05 & 0.00 & 0.04 & 0.00 \\
 &  &  & SVM-CV & 0.97 & 0.07 & 1.63 & 0.11 & 1.93 & 0.12 \\
 &  &  & SVM-L1 & {\(<0.01\)} & 0.00 & {\(<0.01\)} & 0.00 & {\(<0.01\)} & 0.00 \\
 &  &  & SVM-RFE & 0.07 & 0.01 & 0.08 & 0.02 & 0.08 & 0.02 \\
 &  &  & LogReg-AIC & 0.15 & 0.00 & 0.15 & 0.00 & 0.14 & 0.00 \\
 &  &  & LogReg-BIC & 0.15 & 0.00 & 0.16 & 0.00 & 0.15 & 0.00 \\
\midrule
20 & 10 & 0.35 & Proposed-HL & 20.23 & 2.38 & 13.04 & 1.03 & 9.20 & 1.30 \\
 &  &  & Proposed-LS & 6.31 & 0.94 & 4.21 & 0.59 & 3.55 & 0.78 \\
 &  &  & SVM-CV & 60.72 & 9.39 & 99.16 & 38.16 & 83.67 & 27.69 \\
 &  &  & SVM-L1 & {\(<0.01\)} & 0.00 & {\(<0.01\)} & 0.00 & {\(<0.01\)} & 0.00 \\
 &  &  & SVM-RFE & 0.08 & 0.01 & 0.07 & 0.00 & 0.12 & 0.06 \\
 &  &  & LogReg-AIC & 0.75 & 0.17 & 0.82 & 0.25 & 0.36 & 0.06 \\
 &  &  & LogReg-BIC & 0.47 & 0.06 & 0.89 & 0.29 & 0.46 & 0.06 \\
\midrule
20 & 10 & 0.70 & Proposed-HL & 20.74 & 2.74 & 16.22 & 2.49 & 14.93 & 1.81 \\
 &  &  & Proposed-LS & 5.77 & 0.53 & 5.76 & 0.74 & 6.00 & 0.72 \\
 &  &  & SVM-CV & 88.37 & 17.81 & 167.86 & 35.85 & 169.10 & 32.33 \\
 &  &  & SVM-L1 & {\(<0.01\)} & 0.00 & {\(<0.01\)} & 0.00 & {\(<0.01\)} & 0.00 \\
 &  &  & SVM-RFE & 4.25 & 4.12 & 0.36 & 0.30 & 0.07 & 0.01 \\
 &  &  & LogReg-AIC & 0.85 & 0.27 & 0.59 & 0.11 & 0.73 & 0.18 \\
 &  &  & LogReg-BIC & 0.59 & 0.23 & 0.53 & 0.09 & 1.59 & 0.74 \\
\bottomrule
\end{tabular}
\end{table}
}

The results in Figures~\ref{fig:p10_θ5_corr_0.35-results}--\ref{fig:p20_θ10_corr_0.7-results} show that Proposed-HL delivered high feature selection accuracy with 20 candidate features under high and moderate noise regardless of the correlation level.
Proposed-LS attained feature selection accuracy comparable to or better than that of the other SVM-based methods under low noise at the lower correlation level.
For Proposed-LS, the consistency between its LS-SVM-based upper- and lower-level criteria may have contributed to this result.

We next examine the computational efficiency of the feature selection methods.
Table~\ref{tab:computation-time} reports the mean computation times with standard errors for each method, excluding the time required to search over the candidate values of \(\gamma\) and \(C\).
For \(p=10\) and \(\theta=5\), both proposed methods were faster than the other MIO-based methods and were approximately \(16.0\)--\(48.3\) times faster than SVM-CV.
For \(p=20\) and \(\theta=10\), Proposed-LS was faster than Proposed-HL in all settings because the number of constraints in Proposed-LS does not depend on the number of data instances, unlike in Proposed-HL.
Compared with SVM-CV, Proposed-HL and Proposed-LS were approximately \(3.0\)--\(11.3\) and \(9.6\)--\(29.1\) times faster, respectively, but slower than the other methods because of the additional burden of embedding \(K\)-fold cross-validation in the optimization problem.

\FloatBarrier
\section{Conclusion}\label{sec:conclusion}

We proposed computationally efficient cross-validation-based MIO formulations for feature selection in linear SVM classification.
We formulated the feature selection problem as a bilevel MIO problem, where the lower-level problem fits fold-specific LS-SVM classifiers and the upper-level problem selects features according to a criterion evaluated on the validation sets.
By exploiting the optimality conditions of LS-SVM, we reformulated this bilevel problem as single-level MIO problems.
We proposed two formulations: Proposed-HL, which minimizes the validation hinge loss, and Proposed-LS, which minimizes the validation residual sum of squares.

Proposed-HL delivered high feature selection accuracy with 20 candidate features under high and moderate noise regardless of the correlation level.
Proposed-LS attained feature selection accuracy comparable to or better than that of the other SVM-based methods under low noise at the lower correlation level, where the consistency between its LS-SVM-based upper- and lower-level criteria may have worked effectively.
Both proposed methods were faster than the other MIO-based methods with 10 candidate features and were substantially faster than SVM-CV with 20 candidate features, and Proposed-LS tended to be faster than Proposed-HL.

This study has several limitations.
First, our experiments were limited to synthetic datasets, and the effectiveness of our methods on real-world datasets remains to be examined.
Second, the computation times of our methods increased substantially from \(p=10\) to \(p=20\), which suggests that solving larger problems will be computationally demanding.
Third, our formulations are restricted to linear SVMs.
In view of these limitations, future research directions include validating our methods on real-world datasets, extending our formulations to nonlinear kernel SVMs, developing fast exact algorithms, and designing heuristic algorithms that can efficiently find high-quality solutions to larger problems.

\section*{Declarations}

\noindent\textbf{Funding.}
This work was supported by JSPS KAKENHI Grant Number JP25K01447.

\noindent\textbf{Competing interests.}
Ryuta Tamura is an employee of Gurobi Japan Co., Ltd., an affiliate of Gurobi Optimization, LLC, the developer of Gurobi Optimizer used in this study.
The other authors declare that they have no competing interests.

\noindent\textbf{Ethics declaration.}
Not applicable.

\noindent\textbf{Consent to participate.}
Not applicable.

\noindent\textbf{Consent to publish.}
Not applicable.

\noindent\textbf{Clinical trial number.}
Not applicable.

\noindent\textbf{Data availability.}
This study is based on synthetic datasets generated for numerical experiments.

\noindent\textbf{Code availability.}
The source code used in this study is available on GitHub: \href{https://github.com/masa-tichita/cv-lssvm-feature-selection}{\mbox{\nolinkurl{https://github.com/}}\allowbreak\nolinkurl{masa-tichita/cv-lssvm-feature-selection}}.

\backmatter

\bibliography{sn-bibliography}

\end{document}